\documentclass[12pt]{amsart}
\usepackage[utf8]{inputenc}
\usepackage{amsmath}
\usepackage{graphics}
\usepackage{amsthm, multicol, wrapfig}
\usepackage{bm}
\usepackage{amssymb}
\usepackage{setspace, tikz, float,graphicx, tikz-cd}
\usepackage{mathrsfs}
\usepackage{mathabx}
\usepackage{interval}
\usepackage{amsfonts,latexsym,amscd,euscript,graphicx}

\usepackage{framed, mathtools}
\usepackage{fullpage}
\usepackage{hyperref}
\usepackage[margin=0.8 in]{geometry}
\usepackage{setspace}
\usepackage{enumerate}
\newtheorem{thm}{Theorem}
\newtheorem{lemma}[thm]{Lemma}
\newcommand{\N}{{\mathbb N}}

   \hypersetup{colorlinks=true,citecolor=blue,urlcolor =black,linkbordercolor={1 0 0}}

\allowdisplaybreaks[1]
\usepackage{thmtools}
\declaretheoremstyle[notefont=\bfseries,notebraces={}{},%
    headpunct={},postheadspace=1em,spaceabove=0.5em,spacebelow=0.5em]{mystyle}
\declaretheorem[style=mystyle,numbered=no,name=Theorem]{thm-hand}
\declaretheorem[style=mystyle,numbered=no,name=Conjecture]{conj-hand}
\declaretheorem[style=mystyle,numbered=no,name=Definition.]{defn-hand}
\declaretheorem[style=mystyle,numbered=no,name=Theorem.]{thm-no}
\declaretheorem[style=mystyle,numbered=no,name=Conjecture.]{conj-no}
\declaretheorem[style=mystyle,numbered=no,name=Lemma.]{lemma-no}
\declaretheorem[style=mystyle,numbered=no,name=Lemma]{lemma-hand}
\makeatletter
\def\resetMathstrut@{%
  \setbox\z@\hbox{%
    \mathchardef\@tempa\mathcode`\[\relax
    \def\@tempb##1"##2##3{\the\textfont"##3\char"}%
    \expandafter\@tempb\meaning\@tempa \relax
  }%
  \ht\Mathstrutbox@\ht\z@ \dp\Mathstrutbox@\dp\z@}
\makeatother
\begingroup
  \catcode`(\active \xdef({\left\string(}
  \catcode`)\active \xdef){\right\string)}
  \endgroup
\mathcode`(="8000 \mathcode`)="8000

\newtheorem{conj}[thm]{Conjecture}
\newtheorem{prop}[thm]{Proposition}

\theoremstyle{definition}
\newtheorem{defn}[thm]{Definition}

\usepackage{listings}
\theoremstyle{remark}
\newtheorem*{remark}{Remark}

\newcommand{\Z}{\mathbb Z}

\newcommand{\Tn}{T_n}
\newcommand{\Tnn}{T_{n+1}}
\newcommand{\tildeTn}{\widetilde{\Tn}}
\newcommand{\tildeTnn}{\widetilde{\Tnn}}

\newcommand{\Un}{U_n}
\newcommand{\Unn}{U_{n+1}}
\newcommand{\tildeUn}{\widetilde{\Un}}
\newcommand{\tildeUnn}{\widetilde{\Unn}}

\newcommand{\An}{\mathcal{A}_n}
\newcommand{\Ann}{\mathcal{A}_{n+1}}
\newcommand{\Bn}{\mathcal{B}_n}
\newcommand{\Bnn}{\mathcal{B}_{n+1}}

\newcommand{\nc}{\newcommand}

\nc{\on}{\operatorname}
\nc{\Spec}{\on{Spec}}
\vspace{-15mm}
\author{Richard Moy}
\thanks{Willamette University, Salem OR. Email: \texttt{rmoy@willamette.edu}}
\author{Mehtaab Sawhney}
\thanks{Massachusetts Institute of Technology, Cambridge MA. Email: \texttt{msawhney@mit.edu}}
\author{David Stoner}
\thanks{Harvard University, Cambridge MA. Email: \texttt{dstoner@college.harvard.edu}}
\date{\today}

\begin{document}
\title{Characters of Independent Stanley Sequences}
\begin{abstract}
Odlyzko and Stanley introduced a greedy algorithm for constructing infinite sequences with no 3-term arithmetic progressions when beginning with a finite set with no 3-term arithmetic progressions. The sequences constructed from this procedure are known as \textit{Stanley sequences} and appear to have two distinct growth rates which dictate whether the sequences are structured or chaotic. A large subclass of sequences of the former type is independent sequences, which have a self-similar structure. An attribute of interest for independent sequences is the character. In this paper, building on recent progress, we prove that every nonnegative integer $\lambda\not\in\{1,3,5,9,11,15\}$ is attainable as the character of an independent Stanley sequence, thus resolving a conjecture of Rolnick.
\end{abstract}

\maketitle


\section{Introduction}\label{sect:intro}
Let $\N_0$ denote the set of nonnegative integers. A subset of $\N_0$ is called $\ell$-free if it contains no arithmetic progressions (APs) with $\ell$-terms. We say a subset, or sequence of elements, of $\N_0$ is free of arithmetic progressions if it is $3$-free. In 1978, Odlyzko and Stanley \cite{OS78} used a greedy algorithm (see Definition \ref{def:greedy}), further generalized in \cite{ELRSS99}, to produce AP-free sequences. Their algorithm produced sequences with two distinct growth rates -- those which are highly structured (Type I) and those which are seemingly random (Type II). These classes of Stanley sequences will be more precisely defined in Conjecture \ref{conj:growth}.

\begin{defn}\label{def:greedy}
Given a finite $3$-free set $A=\{a_0,\dots,a_n\}\subset\N_{0}$, the \textit{Stanley sequence generated by $A$} is the infinite sequence $S(A)=\{a_0,a_1,\dots\}$ defined by the following recursion. If $k\ge n$ and $a_0<\dots<a_k$ have been defined, let $a_{k+1}$ be the smallest integer $a_{k+1}>a_k$ such that $\{a_0,\dots, a_k\}\cup\{a_{k+1}\}$ is $3$-free. Though formally one writes $S(\{a_0,\dots,a_n\})$, we will frequently use the notation $S(a_0,\dots,a_n)$ instead.
\end{defn}

\begin{remark}
Without loss of generality, we may assume that every Stanley sequence begins with $0$ by shifting the sequence.
\end{remark}

In Rolnick's investigation of Stanley sequences \cite{R17}, he made the following conjecture about the growth rate of the two types of Stanley sequences.

\begin{conj}\label{conj:growth}
Let $S(A)=(a_n)$ be a Stanley sequence. Then, for all $n$ large enough, one of the following two patterns of growth is satisfied:
\begin{itemize}
\item{
Type I: $\alpha\slash 2\le \liminf{a_n\slash n^{\log_2(3)}}\le\limsup{a_n\slash n^{\log_2(3)}}\le \alpha$ for some constant $\alpha$, or
}
\item{
Type II: $a_n=\Theta(n^2\slash \ln(n))$.
}
\end{itemize}
\end{conj}

Though Type II Stanley sequences are mysterious, a great deal of progress has been made in classifying Type I sequences \cite{MR16}. In \cite{R17}, Rolnick introduced the concept of the \textit{independent Stanley sequence} which follow Type I growth and are defined as follows:
\begin{defn}\label{def:independent}
A Stanley sequence $S(A)=(a_n)$ is \textit{independent} if there exist constants $\lambda=\lambda(A)$ and $\kappa=\kappa(A)$ such that for all $k\ge \kappa$ and $0\le i<2^k$, we have 
\begin{itemize}
\item{
$a_{2^{k}+i}=a_{2^k}+a_i$
}
\item{
$a_{2^k}=2a_{2^k-1}-\lambda+1.$
}
\end{itemize}
\end{defn}
This definition's importance stems from the fact that a suitable generalization of such sequences, known as \textit{psuedomodular} sequences, defined by Moy and Rolnick in \cite{MR16}, appear to encompass all Type I sequences, and many natural Type I sequences fall into this class.
\begin{defn}
Given a Stanley sequence $S(A)$, we define the \textit{omitted set} $O(A)$ to be the set of nonnegative integers that are neither in $S(A)$ nor are covered by $S(A)$. For $O(A)\ne \varnothing$, we let $\omega(A)$ denote the largest element of $O(A)$. 
\end{defn}
\begin{remark}
The only Stanley sequence $S(A)$ with $O(A)=\varnothing$ is $S(0)$.
\end{remark}
Using this definition, one can show the following lemma.

\begin{lemma}[Lemma 2.13, \cite{R17}]\label{lemma:omitted}
If $S(A)$ is independent, then $\omega(A)<\lambda(A)$.
\end{lemma}

Lemma \ref{lemma:omitted} will be used (sometimes implicitly) in the proofs of Lemmas \ref{lemma:T} and \ref{lemma:U}. The constant $\lambda$ from Definition \ref{def:independent} is called the \textit{character}, and it is easy to show that $\lambda\ge 0$ for all independent Stanley sequences. If $\kappa$ is taken as small as possible, then $a_{2^\kappa}$ is called the \textit{repeat factor}. Informally, $\kappa$ is the point at which the sequence begins its repetitive behavior. Rolnick and Venkataramana proved that every sufficiently large integer $\rho$ is the repeat factor of some independent Stanley sequence \cite{RV15}.

Rolnick also made a table \cite{R17} of independent Stanley sequences with various characters $\lambda\ge 0$. He found Stanley sequences with every character up to $75$ with the exception of those in the set $\{1,3,5,9,11,15\}$. In light of his observations, he made the following conjecture.

\begin{conj}[Conjecture 2.15, \cite{R17}]\label{conj:allchar}
The range of the character function is exactly the set of nonnegative integers $n$ that are not in the set $\{1,3,5,9,11,15\}$.
\end{conj}

There has been much recent progress towards verifying this conjecture.

\begin{thm}[Theorem 1.10, \cite{M17}]
Let $S(A)$ be an independent Stanley sequence where $A$ is a finite $3$-free subset of $\N_0$. Then $\lambda(A)\not\in \{1,3,5,9,11,15 \}$.
\end{thm}

\begin{thm}[Theorem 1.5, \cite{S17}]
All nonnegative integers $\lambda\equiv 0\bmod 2$ and $\lambda\nequiv 244\bmod 486$ can be achieved as characters of independent Stanley sequences.
\end{thm}

The aim of this paper is to prove Conjecture \ref{conj:allchar}. 
\begin{thm}\label{thm:main1}
All nonnegative integers $\lambda\not\in\{1,3,5,9,11,15\}$ can be achieved as characters of independent Stanley sequences.
\end{thm}

The proof of this theorem may be found at the end of Section \ref{sect:odd}. In order to prove this result, we will utilize the theory of modular sets developed in \cite{MR16} and near-modular sets developed in \cite{S17}. Modular sequences comprise a class of Stanley sequences of Type I which contains all independent Stanley sequences as a strictly smaller subset. 


\begin{defn}
Let $A$ be a set of integers and $z$ be an integer. We say that $z$ is \emph{covered} by $A$ if there exist $x,y\in A$ such that $x<y$ and $2y-x=z$. We frequently say that $z$ is covered by $x$ and $y$.

Suppose that $N$ is a positive integer. If $x,y,z\in\Z$, we say they form an arithmetic progression modulo $N$, or a \emph{mod-AP}, if $2y-x\equiv z\pmod N$. 

Suppose again that $N$ is a positive integer and $A\subseteq \Z$. Then we say that $z$ is \emph{covered by $A$ modulo $N$}, or \emph{mod-covered}, if there exist $x,y\in A$ with $x\le y$ such that $x,y,z$ form an arithmetic progression modulo $N$.
\end{defn}

\begin{defn}
Fix an $N\ge 1$. Suppose the set $A\subset\{0,\ldots,N-1\}$ containing $0$ is 3-free modulo $N$ and all integers $\ell$ are covered by $A$ modulo $N$. Then $A$ is said to be a \emph{modular set modulo $N$} and $S(A)$ is said to be a \emph{modular Stanley sequence modulo $N$}.
\end{defn}

Observe that the modulus $N$ of a modular Stanley sequence is analogous to the repeat factor $\rho$ of an independent Stanley sequence. One can make this statement more precise in the following proposition:

\begin{prop}[Proposition 2.3, \cite{MR16}]
Suppose $A$ is a finite subset of the nonnegative integers and suppose $S(A)$ is an independent Stanley sequence with repeat factor $\rho$. Then $S(A)$ is a modular Stanley sequence modulo $3^n\cdot \rho$ for some integer $n\ge 0$.
\end{prop}

We end this section with the definition of near-modular sets, a slight generalization of modular sets introduced by Sawhney \cite{S17}, and a product on such sets.
\begin{defn}
Fix $N\in\N$. A set $A$ is said to be \emph{near-modular} $\bmod$ $N$ if $0\in A$, $A$ is 3-free mod $N$, and every integer $\ell$ is mod-covered by $A$.
\end{defn}
Note that the only difference between a near-modular set and a modular set defined earlier is that there is no size restriction on the maximum element of a near-modular set. We now define a product on near-modular sets to compose such sets. This operation is identical to that for modular sets given in Moy and Rolnick in \cite{MR16} and we include the proof for completeness. For notational purposes, $X+Y=\{x+y~|~x\in X, y\in Y\}$ and $c\cdot X=\{cx~|~ x\in X\}$. 

\begin{lemma}\label{lemma:tensor}
Suppose that $A$ is a near-modular set $\bmod$ $N$ and $B$ is a near-modular set $\bmod$ $M$ then $A\otimes B=A+N\cdot B$ is a near modular set $\bmod$ $MN$. We henceforth refer to $\otimes$ as the product of two modular sets.
\end{lemma}

\begin{proof}
We show that each of the conditions outlined above is satisfied by $A+N\cdot B$.
\begin{itemize}
    \item Since $0\in A, B$, it follows that $0+N\cdot(0)=0\in A\otimes B$.
    \item Suppose that $x_A+Nx_B, y_A+Nb_2,$ and $z_A+Nz_B$ are in arithmetic progression $\bmod$ $MN$ with $x_A,y_A,z_A\in A$ and $x_B,y_B,z_B\in B$. It follows that $x_A, y_A,$ and $z_A$ are in arithmetic progression $\bmod$ $N$, and therefore $x_A=y_A=z_A$ as $A$ is a near-modular set. Thus it follows that $x_B, y_B,$ and $z_B$ form an arithmetic progression $\bmod$ $M$. Since $B$ is a near-modular set as well it follows that $x_A=y_B=z_B$ and the result follows.
    \item It obviously suffices to prove this for $0\le \ell\le MN-1$. Observe that $\ell= (\ell\bmod N)+N\lfloor \frac{\ell}{N}\rfloor$. There exist $x_A,y_A\in A$ such that $2y_A-x_A\equiv \ell\bmod N$ and $y_A\ge x_A$. Therefore it follows that $2y_A-x_A\equiv \ell+C\cdot N\bmod MN$. Now there exists $2y_B-x_B
    \equiv \lfloor \frac{\ell}{N}\rfloor-C\bmod M$ so that $2(y_A+Ny_B)-(x_A+Nx_B)\equiv \ell\bmod MN$ and the result follows.
\end{itemize}
\end{proof}


\section{Even Characters}\label{sect:even}
In light of the similarity between modular and near-modular sets, we connect the two definitions and show that given a near-modular set with a certain maximal value, we can produce a Stanley sequence with a particular character by constructing a suitable modular set.
\begin{lemma}\label{lemma:nearModtoMod}
Given a near-modular set $L\bmod N$ with maximum element $t$, there exists a modular Stanley sequence with character $\lambda=2t+1-N$. Furthermore if $|L|$ is a power of two, this is an independent Stanley sequence.
\end{lemma}
\begin{proof}
The key idea of this lemma is to repeatedly takes the product a near-modular set with $\{0,1\}$ (which is near-modular $\mod 3$) to reduce a near-modular set to a modular set. In particular let \[L_k=L\otimes \{0,1\}\cdots \otimes\{0,1\}\] where we have taken the product with $\{0,1\}$ $k$ times. Using Lemma \ref{lemma:tensor}, $L_k$ is a modular set modulo $N\cdot 3^k$ provided its maximal element is less than $N\cdot 3^k$. Observe that 
\[\max\{L_k\}=t+N(1+\cdots+3^{k-1})=t+\frac{N(3^k-1)}{2}<t+\frac{N\cdot 3^k}{2},\]
which is less than $N\cdot 3^{k}$ for $k$ sufficiently large.
The character obtained from $S(L_k)$ is easily calculated and we take $(a_i)=S(L_k)$. Note that this sequences is modular by Theorem 2.4 in Moy and Rolnick \cite{MR16}. By construction, $a_{|L_k|-1}=\max{L_k}$ and $a_{|L_k|}=N\cdot 3^k$. It follows that \[\lambda=2a_{|L_k|-1}-a_{|L_k|}+1\]\[=2(t+N(1+\cdots+3^{k-1}))-N\cdot3^k+1\]\[=2t+1-N\] as desired. The final statement that this sequence is independent if $|L|$ is a power of two is essentially by definition.
\end{proof}

Given this procedure of converting near-modular sets to modular sets, we now demonstrate that the existence of certain modular sets implies that all positive even integers occur as characters. Note this lemma can be deduced using techniques in \cite{S17}, however the proof techniques used in Lemma \ref{lemma:evenchar} are used extensively in Section \ref{sect:odd}. 

\begin{lemma}\label{lemma:evenchar}
Suppose that for each positive integer $t$, there exists $A_t$ which is modular $\bmod$ $3^t$ and $\max\{A_t\}=2\cdot 3^{t-1}$. Then each positive even integer is the character of a Stanley sequence.

\end{lemma}
\begin{proof}
We first demonstrate that there exist near-modular sets $A_t^k\bmod 3^t$  with $\max\{A_t^k\}=k\cdot3^{t-1}$ for $k\ge 2$ and $k\nequiv 0\bmod3$. The hypothesis implies the existence of such sets for $k=2$. Since $\gcd(2,3)=1$, observe that $2\cdot A_t$ is a near-modular set $\bmod$ $3^n$ and therefore $2\cdot A_t$ satisfies the condition to be $A_t^{4}$. For $k>4$, note that by adding $3^k$ to the largest element in $A_t^{2}$ and $A_t^{4}$ give $A_t^{5}$ and $A_t^{7}$ and continuing this process inductively gives a near-modular sets $A_t^{k}$ for all $k\ge 2$ and $k\nequiv 0\bmod 3$.
Applying Lemma \ref{lemma:nearModtoMod}, this allows us to obtain all characters of the form
\[2k\cdot 3^{t-1}+1-3^t=(2k-3)3^{t-1}+1.\]
Letting $k$ range over integers $\ge 2$ and $t$ range over all positive integers gives the desired result.
\end{proof}

The remainder of this section is dedicated to constructing the modular sets $A_t$ required by Lemma \ref{lemma:evenchar}. The following sets will be crucial in proving Theorem 9. Let 
$$
T_n:=(\{0,1\}\otimes\{0,2\})^n,
$$

$$
\widetilde{T_n}:=T_n\otimes\left\{0,1\right\},
$$

$$
\An:=(\widetilde{T_n}\backslash\left\{3^{2n}\right\})\cup\{2\cdot 3^{2n}\},
$$

$$
U_n:=\{0,2\}\otimes (\{0,1\}\otimes\{0,2\})^{n-1},
$$

$$
\widetilde{U_n}:=U_n\otimes\{0,1\},
$$
and
$$
\Bn:=(\widetilde{U_n}\backslash\left\{3^{2n-1} \right\})\cup\left\{2\cdot 3^{2n-1} \right\},
$$

where the exponents refer to repeatedly taking the product by the set being exponentiated.
\begin{remark}
Observe that $T_{n+1}=(\{0,1\}\otimes\{0,2\})\otimes T_n$ and $U_{n+1}=(\{0,2\}\otimes\{0,1\})\otimes U_n$. Furthermore, $\Tn,\tildeTn,\Un,\tildeUn$ are modular sets since they are constructed from other modular sets using the $\otimes$ operation.
\end{remark}

\begin{thm}\label{thm:A_t}
For all $t\in \N$ there exists a modular set $A_t$ modulo $3^{t}$ with $|A_t|=2^{t}$ and $\max(A_t)=2\cdot 3^{t-1}$.
\end{thm}

Observe that the $A_t$ in Theorem \ref{thm:A_t} is the same as the one mentioned in Lemma \ref{lemma:evenchar}. The proof of Theorem \ref{thm:A_t} follows from Lemmas \ref{lemma:T} and \ref{lemma:U} below.

\begin{lemma}\label{lemma:T}
For all $n\in\mathbb{N}$, $\An$ is modular set modulo $3^{2n+1}$ with $|\An|=2^{2n+1}$ and $\max(\An)=2\cdot 3^{2n}$.
\end{lemma}

\begin{proof}
We proceed by induction on $n$.
Base Cases: $n=0,1$. A quick calculation shows that $\mathcal{A}_0=\{0,2\}$ is a modular set modulo $3$ with $|\mathcal{A}_0|=2$ and $\max(\mathcal{A}_0)=2$ and that $\mathcal{A}_1=\{0,1,6,7,10,15,16,18\}$ is a modular set modulo $27$ with $|\mathcal{A}_1|=8$ with $\max(\mathcal{A}_1)=18$. 

\medskip

Induction Step: Suppose $n\ge 1$ and that the set $\An$ is modular with modulus $3^{2n+1}$, $|\An|=2^{2n+1}$ and $\max(\An)=2\cdot 3^{2n}$. We wish to prove that $\Ann$ is a modular set with modulus $3^{2n+3}$ with $|\Ann|=2^{2n+3}$ and $\max(\Ann)=2\cdot 3^{2n+2}$. First we show that $\Ann$ is 3-free modulo $3^{2n+3}$. Since $\tildeTnn$ is 3-free modulo $3^{2n+3}$, any mod-AP in $\Ann$ contains the element $2\cdot 3^{2n+2}$.

If there exists a mod-AP in $\Ann$ then there either exist (Case I) $x,y\in \Ann\cap \tildeTnn$ such that $2y\equiv x+2\cdot 3^{2n+2}\bmod 3^{2n+3}$ or (Case II) $x,z\in \Ann\cap \tildeTnn$ such that $x+z\equiv 2\cdot 2\cdot 3^{2n+2}\equiv3^{2n+2}\bmod 3^{2n+3}$.

Case I: If $x\ge 3^{2n+2}$ then $x> 3^{2n+2}$ and $x=3^{2n+2}+\tilde{x}$ where $\tilde{x}\in \Tnn$ by definition of $\tildeTnn$. In this case, we have $2y\equiv\tilde{x}\bmod 3^{2n+3}$, a contradiction with $\tildeTnn$ being 3-free modulo $3^{2n+3}$. If $y>3^{2n+2}$, we obtain a similar contradiction.

If $x<y<3^{2n+2}$, then $0<2y-x<2\cdot 3^{2n+2}$, a contradiction. Finally, if $y<x<3^{2n+2}$, then $-3^{2n+2}\le 2y-x<3^{2n+2}$. However, the left inequality is only obtainable when $x=3^{2n+2}$, a contradiction with $3^{2n+2}\not\in \Ann$.

\medskip

Case II: In this case, $x,z<2\cdot 3^{2n+2}$ and $x,z\ne 3^{2n+2}$. Hence $x+z<4\cdot 3^{2n+2}$ and thus $x+z=3^{2n+2}$ with $x,z\in T_{n+1}$. Therefore, $z,3^{2n+2},3^{2n+2}+x$ is an AP in $\tildeTnn$, a contradiction.

\medskip

We conclude that $\Ann$ is 3-free modulo $3^{2n+3}$.

\medskip\medskip

Now we show that all $z\in \{0,\dots,3^{2n+3}-1\}\backslash \Ann$ are covered by $\Ann$ modulo $3^{2n+3}$. If $z\in \{0,\dots,3^{2n+3}-1\}\backslash \left(\Ann\cup\{3^{2n+2}\}\right)$, then $z$ is mod-covered by $\tildeTnn$. If $z$ is mod-covered by $x<y<z$ with $x,y\in \tildeTnn$, then it is mod-covered by $\Ann$ unless $x$ or $y$ is $3^{2n+2}$.

Therefore, we need to show that elements of the following two forms are mod-covered by $\Ann$.
\begin{itemize}
\item{
Case I: $2\cdot 3^{2n+2}-x$ where $x\in T_{n+1}$ and $x\ne 0$
}
\item{
Case II: $3^{2n+2}+2x$ where $x\in T_{n+1}$
}
\end{itemize}

Case I: In this case $x=9x'+x_0$ with $x'\in T_n$ and $x_0\in\{0,1,6,7\}$ by definition of $T_{n+1}$. If $x'\ne 0$ then $2\cdot 3^{2n}-x'\not\in \An$ and is therefore covered (since $\omega(\An)<\lambda(\An)=3^{2n}+1\le 2\cdot 3^{2n}-x'$) by $\tilde{x},\tilde{y},2\cdot 3^{2n}-x'$ with $\tilde{x}<\tilde{y}$ and $\tilde{x},\tilde{y}\in \An$. Clearly, $\tilde{x},\tilde{y}\ne 3^{2n},2\cdot 3^{2n}$; therefore, $\tilde{x},\tilde{y}\in T_n$. Hence, $9\tilde{x}+x_0<9\tilde{y}$ covers $2\cdot 3^{2n+2}-x$. Unfortunately, our argument fails when $x'=0$ and $x_0\ne 0$. In these three cases, we produce the following 3-term APs: 
\begin{itemize}
\item $15,3^{2n+2}+7,2\cdot 3^{2n+2}-1$
\item $144,3^{2n+2}+69,2\cdot 3^{2n+2}-6$
\item $9,3^{2n+2}+1,2\cdot 3^{2n+2}-7$.
\end{itemize}
These coverings always work because, for all $n\ge 1$, $1,7,9,15,69\in\Tnn$, $144\in \tildeTnn$ and $15<3^{2n+2}+7$, $144<3^{2n+2}+69$, and $9<3^{2n+2}+1$.

\medskip

Case II: In this case $x=9x'+x_0$ with $x'\in T_n$ and with $x_0\in\{0,1,6,7\}$. When $x'=x_0=0$, we have the mod-covering $0,2\cdot 3^{2n+2},3^{2n+2}$. In the case $x'=0$ and $x_0=6,7$, we use the following mod-coverings: $3^{2n+2}+6,3^{2n+2}+9,3^{2n+2}+12$ and $3^{2n+2}+6,3^{2n+2}+10,3^{2n+2}+14$. In the case $x'=0$ and $x_0=1$, observe that $3^{2n}-1$ is covered by $\tilde{x}<\tilde{y}<3^{2n}-1$ in $T_n$. Hence $9\tilde{x}+1,9\tilde{y}+6,3^{2n+2}+2$ covers $3^{2n+2}+2$ with $9\tilde{x}+1,9\tilde{y}+6\in \Ann$.

Otherwise, $x'\ne 0$. Therefore, $3^{2n}+2x'\not\in \An$ and it is mod-covered (and in fact covered) by a 3-term progression $\tilde{x},\tilde{y},3^{2n}+2x'$ with $\tilde{x},\tilde{y}\in \An$. If $\tilde{x},\tilde{y}\ne 2\cdot 3^{2n}$ then $9\tilde{x}<9\tilde{y}+x_0$ covers $x$. 

\medskip

Now suppose $x'\ne 0$ and $\tilde{y}=2\cdot 3^{2n}$. (Clearly $\tilde{x}\ne 2\cdot 3^{2n}$.)

\medskip
\begin{quote}
{\bf Claim:} This only occurs when $x'=2\cdot 3^{2n-1}$.
\\ Proof of Claim: Observe that the greatest element of $\tildeTn$ strictly less than $3^{2n}$ is 
\[
\beta=\sum_{i=0}^{2n-1}{a_i 3^i}\textup{ where } a_i=
\begin{cases}
1& i\textup{ even}\\
2 & i\textup{ odd}\\
\end{cases}
\]
and is in fact the largest element of $\Tn$. Also observe that there are no elements of $T_n$ between $ \beta-2\cdot 3^{2n-1}$ and $2\cdot 3^{2n-1}$ (non-inclusive). 

\medskip

If $x'<2\cdot 3^{2n-1}$, then $x'\le \beta-2\cdot 3^{2n-1}$. Hence, $3^{2n}+2x'\le 3^{2n}+2(\beta-2\cdot 3^{2n-1})<2\cdot 3^{2n}=\tilde{y}$, a contradiction.

\medskip

If $x'\ge 2\cdot 3^{2n-1}$, then $\tilde{x}=2\tilde{y}-(3^{2n}+2\cdot x')\le 3^{2n}+2\cdot 3^{2n-1}.$ Since $x'<3^{2n}$, we know that $x'\le \beta$. Hence $\tilde{x}=2\tilde{y}-(3^{2n}+2\cdot x')\ge 2\tilde{y}-(3^{2n}+2\beta)=3^{2n}+3+1+\sum_{i=1}^{n-1}{2\cdot 3^{2i}}.$ However, this implies that $\tilde{x}\ge 3^{2n}+2\cdot 3^{2n-1}$ by the structure of $\Tn$ and $\tildeTn$. Therefore, $\tilde{x}=3^{2n}+2\cdot 3^{2n-1}$ and $x'=2\cdot 3^{2n-1}$.
\end{quote}
Finally, we have to show that $2\cdot 3^{2n+2}+3^{2n+1}+2x_0$ are mod-covered by $\Ann$ when $x_0\in\{0,1,6,7\}$. We provide the mod-coverings here:
\begin{itemize}
\item $3^{2n+2}+2\cdot 3^{2n+1},2\cdot 3^{2n+2},2\cdot 3^{2n+2}+3^{2n+1}$
\item $3^{2n+2}+10,3^{2n+2}+2\cdot 3^{2n+1}+6,2\cdot 3^{2n+2}+3^{2n+1}+2$
\item $3^{2n+2}+6,3^{2n+2}+2\cdot 3^{2n+1}+9,2\cdot 3^{2n+2}+3^{2n+1}+12$
\item $3^{2n+2}+6,3^{2n+2}+2\cdot 3^{2n+1}+10,2\cdot 3^{2n+2}+3^{2n+1}+14$.
\end{itemize}

We have deduced that $\Ann$ is a modular set modulo $3^{2n+3}$ with $|\Ann|=2^{2n+3}$ and $\max(\Ann)=2\cdot 3^{2n+2}$. Thus we have our result by induction.
\end{proof}

\begin{lemma}\label{lemma:U}
For all $n\in\mathbb{N}$, $\Bn$ is a modular set modulo $3^{2n}$ with $|\Bn|=2^{2n}$ and $\max(\Bn)=2\cdot 3^{2n-1}$.
\end{lemma}

The proof of this proposition is very similar to the proof of Lemma \ref{lemma:T}. We provide its proof for completeness but we omit the details when its proof is identical to the previous lemma.

\begin{proof}
Base Cases: $n=1,2$. A quick calculation shows that $\mathcal{B}_1=\{0,2,5,6\}$ is a modular set modulo $9$ with $|\mathcal{B}_1|=4$ and $\max(\mathcal{B}_1)=2\cdot 3$ and that 
$$\mathcal{B}_2=\{0,2,3,5,18,20,21,23,29,30,32,45,47,48,50,54\}$$
is a modular set modulo $81$ with $|\mathcal{B}_2|=16$ with $\max(\mathcal{B}_2)=54$.

Induction Step: Suppose the set $\mathcal{B}_m$ is modular with modulus $3^{2m}$, $|\mathcal{B}_m|=2^{2m}$ and $\max(\mathcal{B}_m)=2\cdot 3^{2m-1}$ for all $m\le n$ where $n\ge 2$. We wish to prove that $\Bnn$ is a modular set with modulus $3^{2n+2}$ with $|\Bnn|=2^{2n+2}$ and $\max(\Bnn)=2\cdot 3^{2n+1}$. 

The proof that $\Bnn$ is 3-free modulo $3^{2n+2}$ is identical to the analogous proof in the previous lemma.








\medskip\medskip

Now we show that all $z\in \{0,\dots,3^{2n+2}-1\}\backslash \Bnn$ are covered by $\Bnn$ modulo $3^{2n+2}$. If $z\in \{0,\dots,3^{2n+2}-1\}\backslash \left(\Bnn\cup\{3^{2n+1}\}\right)$, then $z$ is mod-covered by $\tildeUnn$. If $z$ is mod-covered by $x<y<z$ with $x,y\in \tildeUnn$, then it is mod-covered by $\Bnn$ unless $x$ or $y$ is $3^{2n+1}$.

Therefore, we need to show that elements of the following two forms are mod-covered by $\Bnn$.
\begin{itemize}
\item{
Case I: $2\cdot 3^{2n+1}-x$ where $x\in \Unn$ and $x\ne 0$
}
\item{
Case II: $3^{2n+1}+2x$ where $x\in \Unn$
}
\end{itemize}

Case I: In this case $x=9x'+x_0$ with $x'\in \Un$ and $x_0\in\{0,2,3,5\}$ by definition of $\Unn$. If $x'\ne 0$ then $2\cdot 3^{2n-1}-x'\not\in \Bn$ and is therefore covered by $\tilde{x},\tilde{y},2\cdot 3^{2n-1}-x'$ with $\tilde{x}<\tilde{y}$ and $\tilde{x},\tilde{y}\in \Bn$. Clearly, $\tilde{x},\tilde{y}\ne 3^{2n-1},2\cdot 3^{2n-1}$; therefore, $\tilde{x},\tilde{y}\in\Un$. Hence, $9\tilde{x}+x_0<9\tilde{y}$ covers $2\cdot 3^{2n+1}-x$. Unfortunately, our argument fails when $x'=0$ and $x_0\ne 0$. In these three cases, we produce the following 3-term APs: 
\begin{itemize}
\item $48,3^{2n+1}+23,2\cdot 3^{2n+1}-2$
\item $45,3^{2n+1}+21,2\cdot 3^{2n+2}-3$
\item $45,3^{2n+1}+20,2\cdot 3^{2n+2}-5$.
\end{itemize}
These coverings work for all $n\ge 1$ because $20,21,23\in\Unn$, $45,48\in\tildeUnn$ and $48<3^{2n+1}+23$, $45<3^{2n+1}+21$, and $45<3^{2n+1}+20$.

\medskip

Case II: In this case $x=9x'+x_0$ with $x'\in \Un$ and with $x_0\in\{0,2,3,5\}$ by definition of $\Unn$. When $x'=x_0=0$, we have the mod-covering $0,2\cdot 3^{2n+1},3^{2n+1}$. In the case $x'=0$ and $x_0=2$, we have the covering $3^{2n+1}+2,3^{2n+1}+3,3^{2n+1}+4$.

\medskip
When $x'=0$ and $x_0\in\{3,5\}$, we require a different argument. Consider $\alpha=3^{2n-3}-1$. If $\alpha\in U_{n-1}$, then $81\alpha+3,81\alpha+45,3^{2n+1}+6$ and $81\alpha+3,81\alpha+47,3^{2n+1}+10$ provide the necessary coverings. If $\alpha\not\in U_{n-1}$, then it is covered by $\tilde{x}<\tilde{y}<\alpha$ modulo $3^{2n-1}$. Therefore, $81\tilde{x}+3,81\tilde{y}+45,3^{2n+1}+6$ and $81\tilde{x}+3,81\tilde{y}+47,3^{2n+1}+10$ provide the necessary coverings.

\medskip

Otherwise, $x'\ne 0$. Therefore, $3^{2n-1}+2x'\not\in \Bn$ and it is mod-covered (and in fact covered) by a 3-term progression $\tilde{x},\tilde{y},3^{2n-1}+2x'$ with $\tilde{x},\tilde{y}\in \Bn$. If $\tilde{x},\tilde{y}\ne 2\cdot 3^{2n-1}$ then $9\tilde{x}<9\tilde{y}+x_0$ covers $z$. 

\medskip

Now suppose $x'\ne 0$ and $\tilde{y}=2\cdot 3^{2n-1}$. (Clearly $\tilde{x}\ne 2\cdot 3^{2n-1}$.)

\medskip
\begin{quote}
{\bf Claim:} This case can only occur when $x'=2\cdot 3^{2n-2}$.

Proof of Claim: The proof is identical to the analogous proof in the previous lemma.
\end{quote}

We still have to show that $2\cdot 3^{2n+1}+3^{2n}+2x_0$ are mod-covered by $\Bnn$ when $x_0\in\{0,2,3,5\}$. We provide the mod-coverings here:
\begin{itemize}
\item $3^{2n+1}+2\cdot 3^{2n},2\cdot 3^{2n+1},2\cdot 3^{2n+1}+3^{2n}$
\item $3^{2n+1}+32,3^{2n+1}+2\cdot 3^{2n}+18,2\cdot 3^{2n+1}+3^{2n}+4$
\item $3^{2n+1}+30,3^{2n+1}+2\cdot 3^{2n}+18,2\cdot 3^{2n+1}+3^{2n}+6$
\item $3^{2n+1}+30,3^{2n+1}+2\cdot 3^{2n}+20,2\cdot 3^{2n+1}+3^{2n}+10$.
\end{itemize}

We have deduced that $\Bnn$ is a modular set modulo $3^{2n+2}$ with $|\Bnn|=2^{2n+2}$ and $\max(\Bnn)=2\cdot 3^{2n+1}$. Thus we have our result by induction.
\end{proof}

\section{Odd Characters}\label{sect:odd}

For the odd character case of Theorem \ref{thm:main1}, we first show that all $\lambda\not\equiv 1\bmod  30$ are attainable as characters of independent Stanley sequences. Using the construction from the previous section, we then show that all odd character values are obtainable with the possible exception of a pair of exponential families. These exceptional families of characters are then shown to be attainable using near-modular sets $\bmod$ $28$.
\begin{lemma}\label{lemma:mod30}
For all $\lambda\ge 61$, $\lambda\equiv 1\bmod 2$, and $\lambda\nequiv 1\bmod 30$, there exists an independent Stanley sequence with character $\lambda$.
\end{lemma}
\begin{proof}
In the appendix, we give near-modular sets $\bmod$ $30$ with maximum elements $46\le \ell\le 59$ and $61\le \ell\le 74$. Note that incrementing the maximum value of each of these sets by $30$, leaves unchanged the property that the set is near-modular $\bmod$ $30$. Therefore it follows that there is a near-modular set $\bmod$ $30$ with maximum element $t$ for $t\ge 45$ and $t\not\equiv 0 \bmod 15$.
Lemma \ref{lemma:nearModtoMod} then implies the existence of Stanley sequences with characters
\[\lambda=2t-30+1=2(t-15)+1\] and the result follows by ranging over all $t\ge 45$ and $t\not\equiv 0\bmod 15$.
\end{proof}
\begin{lemma}\label{lemma:expfamily}
Suppose that $\lambda\equiv 1\bmod 30$ and is not of the form $10\cdot 3^n+1$ or $20\cdot 3^n+1$ for $n\in\mathbb{N}$. Then there exists an independent Stanley sequence with character $\lambda$.
\end{lemma}
\begin{proof}
We first construct families $C_n,D_n,E_n,F_n$ of near-modular sets. Let $R_n$ for $n\ge 0$ denote any family of modular sets such that $R_n$ is modular mod $3^{n+1}$ with maximum element $2\cdot 3^n$. The existence of $R_n$ is guaranteed by Theorem 16. Then
\begin{itemize}
\item Let $C_n=R_n\otimes \{0, 7, 9, 16\}$. The maximum element of $C_n$ is $2\cdot 3^n+16\cdot 3^{n+1}=50\cdot 3^n$, and $C_n$ is a near-modular set mod $10\cdot 3^{n+1}$.
\item Let $R_n'$ denote the set obtained by doubling every element in $R_n$ then increasing the maximum element by $3^{n+1}$. Let $D_n=R'_n\otimes \{0, 7, 9, 16\}$. The maximum element of $D_n$ is $7\cdot 3^n+16\cdot 3^{n+1}=55\cdot 3^n$, and $D_n$ is a near-modular set mod $10\cdot 3^{n+1}$.
\item Let $R_n''$ denote the set obtained by increasing the maximum element in $R_n$ by $3^{n+2}$. Let $E_n=R''_n\otimes \{0, 1,7,8\}$. The maximum element of $E_n$ is $3^{n+2}+2\cdot 3^n+8\cdot 3^{n+1}=35\cdot 3^n$, and $E_n$ is a near-modular set mod $10\cdot 3^{n+1}$.
\item Let $R_n'''$ denote the set obtained by multiplying every element in $R_n$ by $8$. Let $F_n=R_n''''\otimes \{0,1,7,8\}$. The maximum element of $F_n$ is $16\cdot 3^n+8\cdot 3^{n+1}=40\cdot 3^n$, and $F_n$ is a near-modular set mod $10\cdot 3^{n+1}$.
\end{itemize}
Now we consider which characters are obtained by $C_n$, possibly adding multiples of $10\cdot 3^{n+1}$ to the largest element. For any nonnegative integer $k$, we obtain the character
\[2(50\cdot 3^n+10k\cdot 3^{n+1})+1-10\cdot 3^{n+1}=(70+60k)3^n+1.\]
Analogously, the possible characters obtained by $D_n, E_n, F_n$, respectively, are:
\[2(55\cdot 3^n+10k\cdot 3^{n+1})+1-10\cdot 3^{n+1}=(80+60k)3^n+1,\]
\[2(35\cdot 3^n+10k\cdot 3^{n+1})+1-10\cdot 3^{n+1}=(40+60k)3^n+1,\]
and
\[2(40\cdot 3^n+10k\cdot 3^{n+1})+1-10\cdot 3^{n+1}=(50+60k)3^n+1.\]
These four sets together yield all $\lambda$ of the form $10t\cdot 3^n+1$ where $3\nmid t$ and $t\notin\{1, 2\}$. Letting $n$ range over the positive integers gives the desired result. 
\end{proof}
\begin{lemma}\label{lemma:mod28}
For all $\lambda\ge 87$, $\lambda\equiv 1\bmod2$, and $\lambda\not\equiv 1\bmod 14$, there exists an independent Stanley sequence with character $\lambda$.
\end{lemma}
\begin{proof}
In the appendix we give a near-modular sets $\bmod$ $28$ with maximum elements $57\le \ell\le 69$ and $71\le \ell\le 83$. If one increments the maximum value of any of these sets by $28$, the property that the set is near-modular $\bmod$ $28$ remains unchanged. It follows that there is a near-modular set $\bmod$ $28$ with maximal element $t$ for all $t\ge 57$ and $t\not\equiv 0\bmod 14$. Therefore, Lemma \ref{lemma:nearModtoMod} implies the existence of Stanley sequences with characters
\[\lambda=2t-28+1=2(t-14)+1\] and the result follows by ranging over all $t\ge 57$ and $t\not\equiv 0\bmod 14$.
\end{proof}
Now we proceed to prove the main result of the paper.
\begin{thm}\label{thm:main2}
Every odd positive integer $\lambda\notin \{1, 3, 5, 9, 11, 15\}$ is the character of some independent Stanley sequence. 
\end{thm}
\begin{proof}
Rolnick proved the existence of such $\lambda$ for $\lambda\le 73$ in \cite{R17}. Now suppose that some odd $\lambda>73$ is not attainable. Then according to Lemmas \ref{lemma:mod30} and \ref{lemma:expfamily}, $\lambda$ is of the form $10\cdot 3^n+1$ or $20\cdot 3^n+1$, and in particular $\lambda\ge 91$. But according to the Lemma \ref{lemma:mod28}, this implies $7|\lambda-1$, which is impossible. Hence the theorem follows.
\end{proof}
\section{Conclusions}\label{sect:conclusion}
This paper has proven Rolnick's character conjecture (Conjecture \ref{conj:allchar}) by showing that each nonnegative integer $\lambda\not\in\{1,3,5,9,11,15\}$ occurs as the character of some independent Stanley sequence. Further investigation is required to determine whether the theory of near-modular sets can be applied to other conjectures of Rolnick \cite{R17}. Although this paper contributes to the understanding of Stanley sequences with Type I growth, proving that any Stanley sequence follows Type II growth would be a significant contribution to the theory.
\section{Acknowledgements}\label{sect:acknowledgements}
The authors would like to thank Joe Gallian and David Rolnick for carefully reading through the manuscript and providing helpful suggestions. The last two authors were supported by the NSF/DMS--1650947 and conducted this research at the Duluth REU. 
\bibliography{stanley}
\bibliographystyle{plain}
\newpage
\section*{Appendix}
\label{sect:appendix}
The desired modular sets$\mod 28$ and$\mod 30$, are given in the left and right tables below respectively.

\begin{table}[!htb]
    \caption{Necessary Near-Modular Sets Modulo $28$ and $30$}
    \begin{minipage}{.5\linewidth}
      \centering
\begin{tabular}[t]{|c||c| } 
\hline
Max Element & Corresponding Set
\\\hline 57&0,5,11,13,16,18,24,57
\\\hline58&0,8,9,12,27,31,39,58
\\\hline59&0,1,9,10,13,32,40,59
\\\hline60&0,9,12,29,31,38,41,60
\\\hline61&011,13,18,24,29,44,61
\\\hline62&0,11,13,23,24,36,47,62
\\\hline63&0,5,13,17,18,30,50,63
\\\hline64&0,13,17,23,30,40,53,64
\\\hline65&0,3,27,36,39,40,58,65
\\\hline66&0,1,9,12,13,32,59,66
\\\hline67&0,8,13,23,47,52,62,67
\\\hline68&0,3,27,30,36,39,65,68
\\\hline69&0,1,3,9,12,32,66,69
\\\hline71&0,5,9,17,20,22,32,71
\\\hline72&0,11,13,18,24,29,33,72
\\\hline73&0,15,23,27,32,38,40,73
\\\hline74&0,5,11,16,24,29,41,74
\\\hline75&0,11,23,24,34,36,41,75
\\\hline76&0,5,11,16,26,31,43,76
\\\hline77&0,11,15,23,26,34,38,77
\\\hline78&0,4,5,9,15,17,48,78
\\\hline79&0,1,3,4,19,22,48,79
\\\hline80&0,5,13,16,18,39,57,80
\\\hline81&0,13,17,23,36,40,58,81
\\\hline82&0,5,11,15,16,20,59,82
\\\hline83&0,15,17,23,38,40,60,83
\\\hline
\end{tabular}
    \end{minipage}%
    \begin{minipage}{.5\linewidth}
      \centering
\begin{tabular}[t]{|c||c| } 
\hline
Max Element & Corresponding Set
\\\hline46&0,7,9,10,17,19,26,46
\\\hline47&0,7,9,16,20,26,29,47
\\\hline48&0,10,13,21,27,31,34,48
\\\hline49&0,7,9,16,17,26,40,49
\\\hline50&0,1,3,4,14,23,41,50
\\\hline51&0,1,3,4,10,13,44,51
\\\hline52&0,4,21,25,27,31,48,52
\\\hline53&0,5,21,26,27,32,48,53
\\\hline54&0,1,3,21,22,28,49,54
\\\hline55&0,1,3,4,21,24,52,55
\\\hline56&0,2,3,5,21,24,53,56
\\\hline57&0,7,9,20,29,36,56,57
\\\hline58&0,1,10,11,17,18,57,58
\\\hline59&0,7,9,16,36,40,57,59
\\\hline61&0,3,7,10,21,24,28,61
\\\hline62&0,9,13,20,22,23,29,62
\\\hline63&0,19,22,23,29,32,40,63
\\\hline64&0,3,14,20,21,23,41,64
\\\hline65&0,3,17,21,24,38,44,65
\\\hline66&0,7,9,10,17,19,46,66
\\\hline67&0,3,19,21,24,40,46,67
\\\hline68&0,1,7,10,11,18,47,68
\\\hline69&0,3,19,22,23,32,50,69
\\\hline70&0,7,8,11,17,18,51,70
\\\hline71&0,3,20,21,24,34,53,71
\\\hline72&0,3,10,19,22,39,53,72
\\\hline73&0,3,11,14,21,40,54,73
\\\hline74&0,3,10,13,21,41,54,74
\\\hline
\end{tabular}
    \end{minipage} 
\end{table}

\end{document}